\documentclass[12pt]{amsart}
\usepackage[left=3cm, right=3cm, top=2.5cm, bottom=2.5cm]{geometry}

\usepackage{amscd}
\usepackage{amsmath, amssymb, comment}
\usepackage{amsfonts}
\usepackage{enumerate}
\usepackage{color}
\usepackage{mathtools}

\usepackage{hyperref}
\usepackage{url}
\newcommand{\de}{\partial}

\newcommand{\dbar}{\overline{\partial}}

\newcommand{\ov}[1]{\overline{#1}}

\newcommand{\ddbar}{\sqrt{-1} \partial \overline{\partial}}

\newcommand{\ti}[1]{\widetilde{#1}}
\newcommand{\vp}{\varphi}

\newcommand{\ve}{\varepsilon}
\newcommand{\e}{\varepsilon}

\renewcommand{\e}{\varepsilon}
\newcommand{\osc}{\mathrm{osc}}

\newcommand{\Vol}{\mathrm{Vol}}

\renewcommand{\leq}{\leqslant}
\renewcommand{\geq}{\geqslant}

\newcommand{\be}{\begin{equation}}
\newcommand{\ee}{\end{equation}}

\renewcommand{\exp}{\mathrm{exp}}

\newcommand{\BK}{\mathrm{BK}}

\begin{document}
\newcounter{theor}
\setcounter{theor}{1}
\newtheorem{claim}{Claim}
\newtheorem{theorem}{Theorem}[section]
\newtheorem{lemma}[theorem]{Lemma}
\newtheorem{corollary}[theorem]{Corollary}
\newtheorem{proposition}[theorem]{Proposition}
\newtheorem{question}{Question}[section]
\newtheorem{goal}{Goal}[section]
\newtheorem{definition}[theorem]{Definition}
\newtheorem{remark}[theorem]{Remark}

\numberwithin{equation}{section}

\title[Alternative proof of the gradient estimate]{Alternative proof of the gradient estimate for complex Hessian equations}

\author[J. Chu]{Jianchun Chu}
\address{School of Mathematical Sciences, Peking University, Yiheyuan Road 5, Beijing 100871, People's Republic of China}
\email{jianchunchu@math.pku.edu.cn}

\author[Y. Liu]{Yaxiong Liu}
\address{Mathematical Science Reseach Center, Chongqing University of Technology, No. 69, Hongguang Avenue, Banan District, Chongqing 400054, China.}
\address{
Department of Mathematics, University of Maryland, 4176 Campus Dr, College Park, MD 20742, USA}
\email{ yxliu269@cqut.edu, yxliu238@umd.edu}

\begin{abstract}
For complex Hessian equations on compact K\"ahler manifolds, we give an alternative proof of the gradient estimate without using blow-up argument, and clarify its dependence on the right-hand side. In addition, using the similar approach, the case of K\"ahler manifolds with boundary is also investigated.
\end{abstract}


\maketitle

\section{Introduction}

Let $(X,\omega)$ be an $n$-dimensional compact K\"ahler manifold. The complex Hessian equation can be formulated as follows. For $1\leq k\leq n$, denote by $\sigma_{k}$ the $k$-th elementary symmetric polynomial, i.e. for $\lambda=(\lambda_{1},\cdots,\lambda_{n})\in\mathbb{R}^{n}$,
\[
\sigma_{k}(\lambda) := \sum_{1\leq i_{1}<\cdots <i_{k}\leq n}\lambda_{i_{1}}\cdots\lambda_{i_{k}}.
\]
The $k$-th G{\aa}rding cone is defined as
\[
\Gamma_{k} := \left\{ \lambda\in\mathbb{R}^{n}:\sigma_{j}(\lambda)>0 \ \text{for $1\leq j\leq k$} \right\}.
\]
Let $A^{1,1}(X)$ be the space of smooth real $(1,1)$-forms on $X$. For $\alpha\in A^{1,1}(X)$, denote by $\lambda(\alpha)$ the eigenvalue of $\alpha$ with respect to $\omega$. Define
\[
\sigma_{k}(\alpha) := \sigma_{k}(\lambda(\alpha)) = \binom{n}{k}\,\frac{\alpha^{k}\wedge\omega^{n-k}}{\omega^{n}}
\]
and
\[
\Gamma_{k}(X,\omega) := \left\{ \alpha\in A^{1,1}(X):\lambda(\alpha)(x)\in\Gamma_{k} \ \text{for all $x\in X$} \right\}.
\]
For smooth positive function $f$ on $X$, the complex $k$-Hessian equation can be written as
\begin{equation}\label{CHE}
\begin{cases}
\ \sigma_{k}(\omega+\ddbar\vp) = f, \\[1.6mm]
\ \omega+\ddbar \vp \in \Gamma_{k}(X,\omega), \\[1mm]
\ \sup_{X}\vp = 0.
\end{cases}
\end{equation}
When $k=1$, \eqref{CHE} is the Laplace equation. When $k=n$, \eqref{CHE} is the complex Monge-Amp\`ere equation. Yau \cite{Yau78} proved the existence of solutions and then solved the Calabi's conjecture (see \cite{Calabi57}). This result is known as the Calabi--Yau theorem. When $1<k<n$, \eqref{CHE} can be regarded as a natural generalization of the above two important PDEs. Based on Hou--Ma--Wu's estimate \cite{HMW10}, Dinew--Ko\l odziej \cite{DK17} solved \eqref{CHE}.

A usual approach to solve \eqref{CHE} is the continuity method. About establishing a priori estimates, there is a subtlety: the second order estimate seems to depend on the gradient estimate. Such gradient estimate in the Monge--Amp\`ere case was obtained by B\l ocki \cite{Blocki09} and Guan \cite{Guan} independently. However, it seems that their method does not work for the case $1<k<n$. Instead, Hou--Ma--Wu \cite{HMW10} established the following estimate
\begin{equation}\label{HMW estimate introduction}
\sup_{X}|\de\dbar\vp| \leq C\sup_{X}|\de\vp|^{2}+C,
\end{equation}
and pointed out that such estimate is adapted to the blow-up analysis. Later, Dinew--Ko\l odziej \cite{DK17} proved a Liouville type theorem for $k$-subharmonic functions in $\mathbb{C}^{n}$, and then established the gradient estimate by the blow-up argument.

\begin{theorem}[Hou--Ma--Wu \cite{HMW10}, Dinew--Ko\l odziej \cite{DK17}]\label{gradient estimate}
Let $(X,\omega)$ be an $n$-dimensional compact K\"ahler manifold and $\vp$ be a smooth solution of \eqref{CHE}. Then there exists a constant $C$ depending only on $\|f^{1/k}\|_{C^{2}}$, $k$, $n$ and $(X,\omega)$ such that
\[
\sup_{X}|\de\vp|^{2} \leq C.
\]
\end{theorem}

As stated in \cite[Abstract]{DK17}, since Theorem \ref{gradient estimate} is proved by the blow-up argument, then the constant $C$ is not explicit. In this paper, we give an alternative proof of the gradient estimate without using blow-up argument, and clarify its dependence on $f$.

\begin{theorem}\label{main result}
Let $(X,\omega)$ be an $n$-dimensional compact K\"ahler manifold and $\vp$ be a smooth solution of \eqref{CHE}. For $p>n/k$, set $\gamma:=np/(n-kp)$. There exists a constant $C$ depending only on $p$, $k$, $n$ and $(X,\omega)$ such that
\[
\sup_{X}|\de\vp|^{2} \leq \exp\left(e^{C(1+\|f\|_{L^{p}})^{2\gamma}}\right)+\|f\|_{L^{\infty}}+\sup_{X}|\de f^{1/k}|^{2}+\sup_{X}|\de\dbar f^{1/k}|+1.
\]
\end{theorem}

Let us discuss the proof of Theorem \ref{main result}. In \eqref{HMW estimate introduction}, the dependence of $C$ on $f$ is explicit (Theorem \ref{HMW estimate}). Combining this with the Sobelov embedding and $W^{2,p}$-estimate, the gradient estimate can be reduced to the modulus of continuity estimate (Section \ref{sec:proof of main result}).

The argument of the modulus of continuity estimate (Theorem \ref{modulus of continuity estimate}) is inspired by Cheng-Xu \cite{CX24}. They considered \eqref{CHE} on compact K\"ahler manifolds with non-negative holomorphic bisectional curvature, and established the H\"older estimate by using the sup-convolution approximation (\cite[(1.2)]{CX24}). Here the curvature assumption is to guarantee that the sup-convolution approximation is almost $k$-subharmonic in the viscosity sense. In our setting, the logarithmic sup-convolution approximation \eqref{def of vp ve} is considered instead. The advantage is that such approximation is automatically almost $k$-subharmonic in the viscosity sense without any curvature assumption, while the disadvantage is that only the modulus of continuity estimate can be obtained. However, it is enough for our use.

Recently, Xu \cite{Xu26} established the H\"older estimate (\cite[Corollary 1.1]{Xu26}) for the solution of \eqref{CHE} on compact Hermitian manifolds. Let us compare \cite[Corollary 1.1]{Xu26} with Theorem \ref{modulus of continuity estimate}. Indeed, since the H\"older estimate implies modulus of continuity estimate, then \cite[Corollary 1.1]{Xu26} is stronger than Theorem \ref{modulus of continuity estimate}. However, \cite[Corollary 1.1]{Xu26}, the dependence of constant on the right-hand side seems to be non-explicit. The reason is that \cite[Proposition 5.3]{Xu26} was proved by blow-up argument and so the dependence of constant $C_{B}$ on $B\geq\mathrm{osc}_{M}h$ is non-explicit. In the proof of \cite[Theorem 1.4]{Xu26}, the constant $B$ is chosen as the upper bound of $(\mathrm{osc}_{M}u+Ct)$. Since the estimate of $\mathrm{osc}_{M}u$ depends on the right-hand side, then the dependence of constant in \cite[Theorem 1.4]{Xu26} is non-explicit. This leads to the dependence of constant in \cite[Corollary 1.1]{Xu26} is also non-explicit.

The aim of this paper is to give a gradient estimate with explicit dependence of $f$. We may not apply \cite[Corollary 1.1]{Xu26} directly and try to establish the modulus of continuity estimate (Theorem \ref{modulus of continuity estimate}) instead.

Finally, we also consider the Dirichlet problem for the complex Hessian equation. By generalizing the argument of Theorem \ref{main result}, we provide an alternative proof of the gradient estimate (Theorem \ref{gradient estimate Dirichlet}).

\bigskip

The rest of the paper is organized as follows. In Section \ref{sec:preliminaries}, we collect some a priori estimates which will be used later. In Section \ref{sec:modulus of continuity estimate}, we establish the modulus of continuity estimate. The proof of Theorem \ref{main result} will be given in Section \ref{sec:proof of main result}. In Section \ref{sec:boundary case}, the boundary case is investigated. In Appendix \ref{App:proof of Hou-Ma-Wu's estimate}, we give the proof of Hou--Ma--Wu's estimate for the reader's convenience.

\bigskip

{\bf Acknowledgments:} The first-named author was partially supported by National Key R\&D Program of China 2024YFA1014800 and 2023YFA1009900, and NSFC grants 12471052 and 12271008. Part of this work was carried out while the second-named author was visiting the Institute for Theoretical Sciences at Westlake University, which he would like to thank for the hospitality and support.

\bigskip

{\bf Declaration on the use of AI: }
The present paper is motivated by the authors’ previous joint work \cite{CLM26,CLMZ26}. More precisely, we discussed the explicit gradient estimate in \cite[\S 1.2]{CLM26}, and one of the main obstacles to adapting the approach of \cite{CLMZ26} to the complex Hessian setting was the lack of a refined quantitative gradient estimate. The authors have been interested in this problem since their earlier work, but were unable at that time to overcome the difficulty caused by the curvature assumption in the method of Cheng--Xu \cite{CX24}.

During discussions between the authors and GPT-5.6 Sol concerning this difficulty, GPT-5.6 Sol suggested the logarithmic sup-convolution approximation \eqref{def of vp ve}. The proof of Theorem \ref{stability estimate boundary} was provided by GPT-6 Astra, based in part on references supplied by the authors, including works by Guo--Phong \cite{GP24}, Cheng--Xu \cite{CX24} and Xu \cite{Xu26}.

All mathematical arguments were written by the authors, who take full responsibility for the content of this paper.

\section{Preliminaries}\label{sec:preliminaries}

\subsection{Zeroth order estimate}
In \cite{Hou09}, Hou established the zeroth order estimate depending on $\|f\|_{L^{\infty}}$ by the iteration method (see also \cite{HMW10}). In this paper, we will apply the zeroth order estimate depending on $\|f\|_{L^{p}}$. Such estimate was proved by Dinew-Ko\l odziej \cite{DK14}, and later Guo-Phong \cite{GP24} gave a different PDE approach.

\begin{theorem}[Dinew-Ko\l odziej \cite{DK14}, Guo-Phong \cite{GP24}]\label{zeroth order estimate}
Let $(X,\omega)$ be an $n$-dimensional compact K\"ahler manifold and $\vp$ be a smooth solution of \eqref{CHE}. For $p>n/k$, there exists a constant $C$ depending only on $\|f\|_{L^{p}}$, $p$, $k$, $n$ and $(X,\omega)$ such that
\[
\sup_{X}|\vp| \leq C.
\]
\end{theorem}

\subsection{Second order estimate}

\begin{theorem}[Hou-Ma-Wu \cite{HMW10}]\label{HMW estimate}
Let $(X,\omega)$ be an $n$-dimensional compact K\"ahler manifold and $\vp$ be a smooth solution of \eqref{CHE}. Let $C_{\BK}\geq0$ be a constant such that $R(u,\ov{u},v,\ov{v})\geq-C_{\BK}$ for any unit vector $u,v\in T^{1,0}X$. Then there exists a dimensional constant $A_{n}$ (depending only on $n$) such that
\[
\sup_{X}|\de\dbar \vp| \leq e^{A_{n}(C_{\BK}+1)(\sup_{X}|\vp|+1)}\left(\sup_{X}|\de\vp|^{2}+\|f\|_{L^{\infty}}+\sup_{X}|\de f^{1/k}|^{2}+\sup_{X}|\de\dbar f^{1/k}|+1\right).
\]
\end{theorem}

\begin{proof}
This is \cite[Theorem 1.1]{HMW10} with the dependence of the constant on $f$ made explicit. Since the argument needs to be modified slightly (using the test function of Xu \cite[Proposition 5.2]{Xu26}), we give a sketch of proof in Appendix \ref{App:proof of Hou-Ma-Wu's estimate}.
\end{proof}

\subsection{Stability estimate}
\begin{definition}
For $v\in C(X)$, we say that $\omega+\ddbar v\in\Gamma_{k}(X,\omega)$ in the viscosity sense if for any $x\in X$ and for any upper test function $P$ of $v$ at $x$ (i.e. $P(x)=v(x)$ and $P\geq v$ near $x$), one has
\[
\lambda(\omega+\ddbar P)(x) \in \Gamma_{k}.
\]
\end{definition}

\begin{theorem}[Cheng-Xu \cite{CX24}]\label{stability estimate}
Let $(X,\omega)$ be an $n$-dimensional compact K\"ahler manifold and $\vp$ be a smooth solution of \eqref{CHE}. Suppose that $v\in C(X)$ satisfies $v\leq0$ and
\[
\left(1+\frac{\delta}{2}\right)\omega+\ddbar v \in \Gamma_{k}(X,\omega) \ \ \text{in the viscosity sense for some $\delta>0$}.
\]
For $p>n/k$, set $\gamma:=np/(kp-n)$. There exists a constant $C$ depending only on $p$, $k$, $n$ and $(X,\omega)$ such that
\[
\sup_{X}(v-\vp) \leq s_{0}+C\|f\|_{L^{p}}^{1/k} \cdot
\left( \frac{\|(v-\vp)_{+}\|_{L^{1}}}{s_{0}} \right)^{\frac{1}{2k\gamma}} \ \ \text{for $s_{0}\geq s_{*}$},
\]
where
\[
s_{*} \leq  \max\left\{ 2\delta\|v\|_{L^{\infty}}, \ C\delta^{-2k\gamma}\cdot\|f\|_{L^{p}}^{2\gamma}\cdot\|(v-\vp)_{+}\|_{L^{1}} \right\}.
\]
\end{theorem}

\begin{proof}
This is \cite[Lemma 3.6 and 3.7]{CX24} with the choice of
\[
e^{F} = f^{1/k}, \ \ p_{0} = \frac{kp}{n}, \ \ q_{0} = \frac{kp}{kp-n}, \ \ \mu = \frac{1}{2nq_{0}}, \ \ \beta = 2,
\]
and with the dependence of the constant on $f$ made explicit.
\end{proof}

\begin{remark}
If one follows the argument of Theorem \ref{zeroth order estimate}, it seems to be tedious to trace the dependence of $C$ on $\|f\|_{L^{p}}$. Here for convenience, let us derive this dependence from Theorem \ref{stability estimate}. Choosing $v=0$, $\delta=1$ and $s_{0}=s_{*}$,
\[
\sup_{X}|\vp| \leq C\|f\|_{L^{p}}^{2\gamma}\cdot\|\vp\|_{L^{1}}+C.
\]
By Green's formula, one has $\|\vp\|_{L^{1}}\leq C_{G}$ for some constant $C_{G}$ depending only on $(X,\omega)$, it then follows that
\[
\sup_{X}|\vp| \leq C_{0}\big(1+\|f\|_{L^{p}}\big)^{2\gamma}.
\]
\end{remark}

\begin{remark}\label{K 0 K 1}
For notational convenience and later use, denote
\[
K_{0} := C_{0}\big(1+\|f\|_{L^{p}}\big)^{2\gamma}, \ \
C_{1} := A_{n}(C_{\BK}+1), \ \
K_{1} := \exp\Big(C_{1}C_{0}(1+\|f\|_{L^{p}})^{2\gamma}+C_{1}\Big).
\]
Then $\sup_{X}|\vp| \leq K_{0}$ and
\[
\sup_{X}|\de\dbar \vp| \leq K_{1}\left(\sup_{X}|\de\vp|^{2}+\|f\|_{L^{\infty}}+\sup_{X}|\de f^{1/k}|^{2}+\sup_{X}|\de\dbar f^{1/k}|+1\right).
\]
\end{remark}

\section{Modulus of continuity estimate}\label{sec:modulus of continuity estimate}

In this section, we establish the modulus of continuity estimate.

\begin{theorem}\label{modulus of continuity estimate}
Let $(X,\omega)$ be an $n$-dimensional compact K\"ahler manifold and $\varphi$ be a smooth solution of \eqref{CHE}. For $p>n/k$, there exist constants $\ve_{0}\in(0,1)$ and $C$ depending only on $p$, $k$, $n$ and $(X,\omega)$ such that for $d(x,y)<\ve_{0}$,
\[
|\vp(x)-\vp(y)| \leq CK_{0}^{2}\big(-\log d(x,y)\big)^{-1},
\]
where $K_{0}=C_{0}\big(1+\|f\|_{L^{p}}\big)^{2\gamma}$ is the constant in Remark \ref{K 0 K 1}.
\end{theorem}

\subsection{Logarithmic sup-convolution approximation}

For $\ve\in(0,1)$, define one-variable increasing concave function:
\[
\Psi_{\ve}(s) := \frac{K_{0}}{\log(1+\ve^{-1})}\log\left(1+\frac{s}{\ve^{2}}\right) \ \ s \geq 0,
\]
and the logarithmic sup-convolution approximation:
\begin{equation}\label{def of vp ve}
\vp_{\ve}(x) := \sup_{\xi\in T_{x}X}\Big\{\vp(\exp_{x}\xi)-\Psi_{\ve}(|\xi|_{x}^{2})\Big\}.
\end{equation}

The following lemma shows some elementary properties of $\vp_{\ve}$, which follow from the definition directly.

\begin{lemma}\label{elementary properties}
The following holds:
\begin{enumerate}\setlength{\itemsep}{1mm}
\item[(i)] $-K_{0}\leq\vp\leq \vp_{\ve}\leq0$;
\item[(ii)] each maximizing vector $\xi_{x}$ in \eqref{def of vp ve} satisfies $|\xi_{x}|_{x}^{2}\leq \ve$;
\item[(iii)] if $d(x,y)\leq\ve$, then $\vp(y)\leq \vp_{\ve}(x)+2K_{0}(-\log\ve)^{-1}$.
\end{enumerate}
\end{lemma}

\begin{proof}

Using $\Psi_{\ve}\geq0$ and $\vp\leq0$, and choosing $\xi=0$ in \eqref{def of vp ve}, we obtain (i). For (ii), choosing $\xi=0$ in \eqref{def of vp ve} again,
\[
\vp(\exp_{x}\xi_{x})-\Psi_{\ve}(|\xi_{x}|_{x}^{2}) \geq \vp(\exp_{x}0)-\Psi_{\ve}(0) = \vp(x).
\]
Together with $-K_{0}\leq \vp\leq0$,
\[
\frac{K_{0}}{\log(1+\ve^{-1})}\log\left(1+\frac{|\xi_{x}|_{x}^{2}}{\ve^{2}}\right) = \Psi_{\ve}(|\xi_{x}|_{x}^{2}) \leq -\vp(x) \leq K_{0},
\]
which implies (ii). For (iii), let $\xi\in T_{x}X$ be the vector such that $y=\exp_{x}\xi$. Then $|\xi|_{x}=d(x,y)\leq\ve$ and so
\[
\vp_{\ve}(x) \geq \vp(\exp_{x}\xi)-\Psi_{\ve}(|\xi|_{x}^{2})
\geq \vp(y)-\Psi_{\ve}(\ve^{2})
= \vp(y)-\frac{K_{0}\log 2}{\log(1+\ve^{-1})},
\]
which implies (iii).
\end{proof}

\subsection{$k$-subharmonicity and $L^{1}$ approximation}

\begin{proposition}\label{k-subharmonicity and L1 approximation}
There exist constants $\ve_{0}\in(0,1)$ and $C$ depending only on $k$, $n$ and $(X,\omega)$ such that the following holds. For $\ve\in(0,\ve_{0})$, one has
\begin{enumerate}\setlength{\itemsep}{1mm}
\item[(i)] $(1+K_{0}\theta_{\ve})\omega+\ddbar\vp_{\ve}\in\Gamma_{k}(X,\omega)$ in the viscosity sense, where $\theta_{\ve}:=C(-\log\ve)^{-1}$;
\item[(ii)] $\|\vp_{\ve}-\vp\|_{L^{1}}\leq CK_{0}\ve^{1/4}$.
\end{enumerate}
\end{proposition}

\begin{proof}
For (i), fix $x_{0}\in X$ and let $P\in C^{2}(X)$ be a upper test function of $\vp_{\ve}$ at $x_{0}$, i.e.
\begin{equation}\label{upper test function}
P(x_{0}) = \vp_{\ve}(x_{0}) \ \ \text{and} \ \
P \geq \vp_{\ve} \ \text{near $x_{0}$}.
\end{equation}
Let $\xi_{0}$ be a maximizing vector in \eqref{def of vp ve} and set
\[
y_{0}=\exp_{x_{0}}\xi_{0}, \ \ \ s_{0} = |\xi_{0}|_{x_{0}}^{2}.
\]
Choose normal holomorphic coordinates $(U,\{z^{i}\}_{i}^{n})$ centered at $x_{0}$ such that $B(x_{0},\ve)\subset U$. Following \cite[(2.3) and (2.4)]{CX24}, fix a Hermitian matrix $N=(N_{i\ov{j}})$ such that
\[
\sum_{a,b}N_{i\ov{a}}\,\ov{N_{j\ov{b}}}\,g_{a\ov{b}}(y_{0}) = \delta_{ij} = g_{i\ov{j}}(x_{0}),
\]
and define a local holomorphic map
\[
\Phi(z) := y_{0}+N \cdot z.
\]
It then follows that
\begin{equation}\label{pullback omega}
(\Phi^{*}\omega)(x_{0})=(\Phi^{*}\omega)(0)=\omega(x_{0}).
\end{equation}
By \cite[Lemma 2.9]{CX24}, there exists a smooth vector field
\[
z \mapsto \xi(z) \in T_{z}X
\]
such that
\[
\xi(x_{0}) = \xi_{0}, \ \ \ \exp_{z}(\xi(z)) = \Phi(z).
\]
Set $s(z):=|\xi(z)|_{z}^{2}$. By \eqref{upper test function} and \eqref{def of vp ve},
\[
P \geq \vp_{\ve} \geq \vp\circ\Phi-\Psi_{\ve}(s) \ \ \text{near $x_{0}$}
\]
with the equality holds at $x_{0}$. Then the maximum principle shows
\[
\Big(\ddbar P+\ddbar\Psi_{\ve}(s)\Big)(x_{0}) \geq \ddbar(\vp\circ\Phi)(x_{0})
= (\Phi^{*}\ddbar \vp)(x_{0}).
\]
Together with \eqref{pullback omega},
\begin{equation}\label{omega ddbar P ddbar Psi}
\Big(\omega+\ddbar P+\ddbar\Psi_{\ve}(s)\Big)(x_{0}) \geq \Phi^{*}(\omega+\ddbar \vp)(x_{0}).
\end{equation}
We claim that
\begin{equation}\label{k-subharmonicity claim}
\ddbar\Psi_{\ve}(s)(x_{0}) \leq CK_{0}(-\log\ve)^{-1}\omega(x_{0}).
\end{equation}
Given this claim, setting $\theta_{\ve}:=C(-\log\ve)^{-1}$ and using \eqref{omega ddbar P ddbar Psi}, we obtain
\[
\Big((1+K_{0}\theta_{\ve})\omega+\ddbar P \Big)(x_{0}) \geq \Phi^{*}(\omega+\ddbar \vp)(x_{0}).
\]
Since $\omega+\ddbar \vp\in\Gamma_{k}(X,\omega)$, then
\[
\lambda\Big(\Phi^{*}(\omega+\ddbar\vp)\Big)(x_{0}) = \lambda(\omega+\ddbar \vp)(y_{0}) \in \Gamma_{k}
\]
and so
\[
\lambda\Big((1+K_{0}\theta_{\ve})\omega+\ddbar P\Big)(x_{0})  \in \Gamma_{k}
\]
as required.

\medskip

Next we prove the claim \eqref{k-subharmonicity claim}. Direct calculation shows
\[
\Psi_{\ve}'(s) = \frac{K_{0}}{\log(1+\ve^{-1})}\cdot\frac{1}{(s+\ve^{2})}, \ \
\Psi_{\ve}''(s) = -\frac{K_{0}}{\log(1+\ve^{-1})}\cdot\frac{1}{(s+\ve^{2})^{2}}.
\]
Using $\Psi_{\ve}''(s)<0$,
\[
\ddbar\Psi_{\ve}(s) = \Psi_{\ve}'(s)\ddbar s+\Psi_{\ve}''(s)\sqrt{-1}\de s\wedge\dbar s \leq \Psi_{\ve}'(s)\ddbar s.
\]
Set $s_{0}:=s(0)=|\xi_{0}|_{x_{0}}^{2}$. By \cite[Lemma 2.11]{CX24}, one has
\[
|\nabla\xi|(x_{0})+|\nabla^{2}\xi|(x_{0}) \leq C|\xi_{0}|_{x_{0}}^{2} = Cs_{0},
\]
which implies
\[
s_{i\ov{j}}(x_{0}) = \de_{i}\de_{\ov{j}}g_{k\ov{l}}\cdot\xi^{k}\ov{\xi^{l}}(x_{0})+\delta_{k\ov{l}}\de_{i}\de_{\ov{j}}(\xi^{k}\ov{\xi^{l}})(x_{0})
\leq Cs_{0}.
\]
Then we have $\ddbar s(x_{0})\leq Cs_{0}\omega(x_{0})$ and so
\[
\ddbar\Psi_{\ve}(s)(x_{0}) \leq \frac{K_{0}}{\log(1+\ve^{-1})}\cdot\frac{Cs_{0}}{s_{0}+\ve^{2}} \cdot \omega(x_{0}) \leq CK_{0}(-\log\ve)^{-1}\omega(x_{0}),
\]
as required.

\medskip

For (ii), fix $r_{0}\in(0,1)$ such that for any $x\in X$, $B(x,r_{0})$ is contained in the normal holomorphic coordinate and $\omega=\ddbar\rho$ with $\|\rho\|_{C^{1}}\leq 1$. Since $\omega+\ddbar\vp\in\Gamma_{k}\subset\Gamma_{1}$, then
\[
\Delta (\vp+\rho) = \mathrm{tr}_{\omega}(\omega+\ddbar\vp) > 0.
\]
Applying \cite[Proposition 2.16]{CX24} to $\vp+\rho$, we see that for $r\in(0,r_{0})$,
\[
(\vp+\rho)(x) \leq \frac{1}{\alpha_{2n}r^{2n}}\int_{B(x,r)}(\vp+\rho)\,\omega^{n}+C\big(\|\rho+\vp\|_{L^{\infty}}+1\big)r,
\]
where $\alpha_{2n}$ denotes the volume of unit ball in $\mathbb{R}^{2n}$. By the mean value theorem and \cite[Lemma 2.17]{CX24},
\begin{equation}\label{mean value inequality vp}
\begin{split}
\vp(x) \leq {} & \frac{1}{\alpha_{2n}r^{2n}}\int_{B(x,r)}\vp\,\omega^{n}+\big(C\|\rho\|_{C^{1}}+CK_{0}\big)r \\
\leq {} & \frac{1}{\alpha_{2n}r^{2n}}\int_{B(x,r)}\vp\,\omega^{n}+CK_{0}r.
\end{split}
\end{equation}
Let $\xi_{x}$ be the maximizing vector in \eqref{def of vp ve} and set $y_{x}:=\exp_{x}\xi_{x}$. Then $\vp_{\ve}(x)\leq\vp(y_{x})$. Applying \eqref{mean value inequality vp} to $r=\ve^{1/4}$ and $x=y_{x}$, we obtain
\begin{equation}\label{mean value inequality vp ve}
\vp_{\ve}(x) \leq \vp(y_{x}) \leq \frac{1}{\alpha_{2n}\ve^{n/2}}\int_{B(y_{x},\ve^{1/4})}\vp\,\omega^{n}+CK_{0}\ve^{1/4}.
\end{equation}
By Lemma \ref{elementary properties} (ii), $d(x,y_{x})=|\xi_{x}|_{x}\leq\ve^{1/2}$. Since $\ve^{1/4}>\ve^{1/2}$, then $B(x,\ve^{1/4}-\ve^{1/2})\subset B(y_{x},\ve^{1/4})$. Together with $\vp\leq0$ and \eqref{mean value inequality vp ve},
\[
\vp_{\ve}(x) \leq \frac{1}{\alpha_{2n}\ve^{n/2}}\int_{B(x,\ve^{1/4}-\ve^{1/2})}\vp\,\omega^{n}+CK_{0}\ve^{1/4}.
\]
Integrating both sides on $X$ and using Fubini's theorem,
\begin{equation}\label{int vp ve}
\begin{split}
\int_{X}\vp_{\ve}(x)\,\omega^{n}(x)
\leq {} & \frac{1}{\alpha_{2n}\ve^{n/2}}\int_{X}\left(\int_{B(x,\ve^{1/4}-\ve^{1/2})}\vp(y)\,\omega^{n}(y)\right)\omega^{n}(x)+CK_{0}\ve^{1/4} \\
= {} & \frac{1}{\alpha_{2n}\ve^{n/2}}\int_{X}\left(\int_{B(y,\ve^{1/4}-\ve^{1/2})}\vp(y)\,\omega^{n}(x)\right)\omega^{n}(y)+CK_{0}\ve^{1/4} \\[1mm]
= {} & \int_{X}\vp(y)\cdot\Theta_{\ve}(y)\cdot\omega^{n}(y)+CK_{0}\ve^{1/4},
\end{split}
\end{equation}
where
\[
\Theta_{\ve}(y) := \frac{\Vol\big(B(y,\ve^{1/4}-\ve^{1/2})\big)}{\alpha_{2n}\ve^{n/2}}.
\]
By \cite[Lemma 2.17]{CX24}, one has $|\Theta_{\ve}-1|\leq C\ve^{1/4}$. Together with Lemma \ref{elementary properties} (i) and \eqref{int vp ve}, we compute
\[
\begin{split}
\|\vp_{\ve}-\vp\|_{L^{1}} = {} & \int_{X}\left(\vp_{\ve}-\vp\right)\omega^{n} \leq \int_{X}\vp\left(\Theta_{\ve}-1\right)\omega^{n}+CK_{0}\ve^{1/4} \leq CK_{0}\ve^{1/4},
\end{split}
\]
as required.
\end{proof}

\subsection{Proof of Theorem \ref{modulus of continuity estimate}}

Now we are in a position to prove Theorem \ref{modulus of continuity estimate}.

\begin{proof}[Proof of Theorem \ref{modulus of continuity estimate}]
Set $\ve:=d(x,y)$. We claim that
\begin{equation}\label{modulus of continuity estimate claim}
\begin{split}
\sup_{X}(\vp_{\ve}-\vp) \leq CK_{0}^{2}(-\log\ve)^{-1}.
\end{split}
\end{equation}
Combining this claim with Lemma \ref{elementary properties} (iii),
\[
\begin{split}
\vp(y)-\vp(x) \leq {} & \vp_{\ve}(x)+2K_{0}(-\log\ve)^{-1}-\vp(x) \\[1.6mm]
\leq {} & \sup_{X}(\vp_{\ve}-\vp)+2K_{0}(-\log\ve)^{-1} \\
\leq {} & CK_{0}^{2}(-\log\ve)^{-1}.
\end{split}
\]
Switching $x$ and $y$, we are done.

\medskip

Next we prove the claim \eqref{modulus of continuity estimate claim}. Using Proposition \ref{k-subharmonicity and L1 approximation} (i) and applying Theorem \ref{stability estimate} to $v=\vp_{\ve}$ with $\delta=2K_{0}\theta_{\ve}$, we see that
\begin{equation}\label{modulus of continuity estimate eqn 1}
\sup_{X}(\vp_{\ve}-\vp) \leq s_{0}+C\|f\|_{L^{p}}^{1/k} \cdot
\left( \frac{\|(\vp_{\ve}-\vp)_{+}\|_{L^{1}}}{s_{0}} \right)^{\frac{1}{2k\gamma}} \ \ \text{for $s_{0}\geq s_{*}$},
\end{equation}
where
\[
s_{*} \leq  \max\left\{ 4K_{0}\theta_{\ve}\|\vp_{\ve}\|_{L^{\infty}}, \ C(2K_{0}\theta_{\ve})^{-2k\gamma}\cdot\|f\|_{L^{p}}^{2\gamma}\cdot\|(\vp_{\ve}-\vp)_{+}\|_{L^{1}} \right\}.
\]
By Lemma \ref{elementary properties} (i), Proposition \ref{k-subharmonicity and L1 approximation} (ii) and $\|f\|_{L^{p}}^{2\gamma}\leq K_{0}$, we have
\[
\begin{split}
s_{*} \leq {} & \max\left\{ 4K_{0}\theta_{\ve}^{2}, \ C(2K_{0}\theta_{\ve})^{-2k\gamma}\cdot\|f\|_{L^{p}}^{2\gamma}\cdot(CK_{0}\ve^{1/4}) \right\} \\
\leq {} & \max\left\{ CK_{0}^{2}(-\log\ve)^{-1}, \ CK_{0}^{2-2k\gamma}(\log\ve)^{-2k\gamma}\ve^{1/4} \right\} \\
\leq {} & CK_{0}^{2}(-\log\ve)^{-1}.
\end{split}
\]
Choosing $s_{0}=CK_{0}^{2}(-\log\ve)^{-1}$ in \eqref{modulus of continuity estimate eqn 1}, and using  and Proposition \ref{k-subharmonicity and L1 approximation} (ii) and $\|f\|_{L^{p}}^{1/k}\leq K_{0}$,
\[
\begin{split}
\sup_{X}(\vp_{\ve}-\vp) \leq {} & CK_{0}^{2}(-\log\ve)^{-1}+C\|f\|_{L^{p}}^{1/k} \cdot
\left( CK_{0}^{-1}\ve^{1/4}(-\log\ve) \right)^{\frac{1}{2k\gamma}} \\[-1.6mm]
\leq {} & CK_{0}^{2}(-\log\ve)^{-1}+CK_{0} \cdot
\left( CK_{0}^{-1}\ve^{1/4}(-\log\ve) \right)^{\frac{1}{2k\gamma}} \\[2mm]
\leq {} & CK_{0}^{2}(-\log\ve)^{-1},
\end{split}
\]
as required.
\end{proof}

\section{Proof of Theorem \ref{main result}}\label{sec:proof of main result}

Now we are ready to prove Theorem \ref{main result}.

\begin{proof}[Proof of Theorem \ref{main result}]
Set $Q:=\sup_{X}|\de \vp|$ and let $x_{0}\in X$ be the point such that
\[
|\de \vp|(x_{0}) = Q.
\]
Without loss of generality, we assume that
\[
Q^{2} \geq \|f\|_{L^{\infty}}+\sup_{X}|\de f^{1/k}|^{2}+\sup_{X}|\de\dbar f^{1/k}|+1.
\]
Then Theorem \ref{HMW estimate} and Remark \ref{K 0 K 1} show
\begin{equation}\label{de dbar vp K 1 Q 2}
\sup_{X}|\de\dbar \vp| \leq 4K_{1}Q^{2}.
\end{equation}
By Sobelov embedding and $W^{2,p}$-estimate, there exist constants $r_{0}\in(0,1)$ and $C_{*}>1$ depending only on $(X,\omega)$ such that for $r<r_{0}$, one has
\[
r|\de \vp|(x_{0}) \leq C_{*}\underset{B(x_{0},r)}{\osc}\vp+C_{*} r^{2}\|\Delta \vp\|_{L^{\infty}(B(x_{0},r))}.
\]
Combining this with \eqref{de dbar vp K 1 Q 2},
\[
rQ \leq C_{*}\underset{B(x_{0},r)}{\osc}\vp + 4nC_{*}K_{1}r^{2}Q^{2}.
\]
Without loss of generality, we assume that $(8nC_{*}K_{1}Q)^{-1}<r_{0}$. Choosing $r=(8nC_{*}K_{1}Q)^{-1}$, we obtain
\[
\underset{B(x_{0},r)}{\osc}\vp \geq \frac{1}{16nC_{*}^{2}K_{1}}.
\]
By Theorem \ref{modulus of continuity estimate},
\[
\underset{B(x_{0},r)}{\osc}\vp \leq CK_{0}^{2}(-\log r)^{-1}.
\]
It then follows that
\[
r \geq \exp(-CK_{0}^{2}K_{1}).
\]
Combining this with $r=(8nC_{*}K_{1}Q)^{-1}$,
\[
Q \leq \exp(CK_{0}^{2}K_{1}).
\]
Combining this with the definitions of $K_{0}$ and $K_{1}$ in Remark \ref{K 0 K 1}, we are done.
\end{proof}

\section{The case of K\"ahler manifolds with boundary}\label{sec:boundary case}

In this section, we generalize the argument of Theorem \ref{main result} and give an alternative proof of the gradient estimate for complex Hessian equations (without using blow-up argument) on compact K\"ahler manifolds with boundary. Let $(\Omega,\omega)$ be an $n$-dimensional compact K\"ahler manifold with $\de\Omega\neq\emptyset$. For smooth positive function $f$ on $\Omega$ and $\psi\in C^{\infty}(\de\Omega)$, consider the Dirichlet problem:
\begin{equation}\label{CHE Dirichlet}
\begin{cases}
\ \sigma_{k}(\omega+\ddbar\vp) = f & \mbox{in $\Omega$}, \\[1.2mm]
\ \vp = \psi & \mbox{on $\de\Omega$}, \\[1mm]
\ \omega+\ddbar \vp \in \Gamma_{k}(\Omega,\omega).
\end{cases}
\end{equation}
Suppose that $\underline{\vp}$ be a smooth subsolution of \eqref{CHE Dirichlet}, i.e.
\[
\begin{cases}
\ \sigma_{k}(\omega+\ddbar\underline{\vp})\geq f & \mbox{in $\Omega$}, \\[1mm]
\ \underline{\vp} = \psi & \mbox{on $\de\Omega$}.
\end{cases}
\]
In \cite{CP22}, Collins--Picard established the Hou--Ma--Wu type estimate:
\begin{equation}\label{HMW type estimate Dirichlet}
\sup_{\ov{\Omega}}|\de\dbar\vp| \leq C\sup_{\ov{\Omega}}|\de\vp|^{2}+C,
\end{equation}
and then proved the following gradient estimate by the blow-up argument.

\begin{theorem}[Collins--Picard \cite{CP22}]\label{gradient estimate Dirichlet}
Let $(\Omega,\omega)$ be an $n$-dimensional compact K\"ahler manifold with $\de\Omega\neq\emptyset$, $\vp$ be a smooth solution of \eqref{CHE Dirichlet} and $\underline{\vp}$ be a smooth subsolution of \eqref{CHE Dirichlet}. Then the gradient estimate holds:
\[
\sup_{\ov{\Omega}}|\de\vp| \leq C(f,\underline{\vp},n,k,\Omega,\omega).
\]
\end{theorem}

In the following, we will prove Theorem \ref{gradient estimate Dirichlet} without using blow-up argument, and so the dependence of $C$ on $f$ can be computed concretely. However, it seems to be quite tedious to trace this dependence. Then here we will not pursue this, and only give a sketch of the proof (as the argument is similar to that of the non-boundary case).

To prove Theorem \ref{gradient estimate Dirichlet}, using \eqref{HMW type estimate Dirichlet} and similar argument of Theorem \ref{main result}, it suffices to establish the modulus of continuity estimate. Such estimate follows from the global logarithmic sup-convolution approximation (Proposition \ref{global log sup convolution}), stability estimate (Theorem \ref{stability estimate boundary}) and the similar argument of Theorem \ref{modulus of continuity estimate}.

\medskip

We assume without loss of generality that $\psi\leq0$ (replacing $\psi$ and $\vp$ by $\psi-\sup_{\de\Omega}\psi$ and $\vp-\sup_{\de\Omega}\psi$).
Let $b$ be the smooth function satisfying
\[
\begin{cases}
\ \Delta_{\omega}b = -n & \mbox{in $\Omega$}, \\[1mm]
\ b = \psi & \mbox{on $\de\Omega$}.
\end{cases}
\]
By the maximum principle, we see that
\begin{equation}\label{underline vp vp b}
\begin{cases}
\ \underline{\vp} \leq \vp \leq b \leq 0  & \mbox{in $\Omega$}, \\[1mm]
\ \underline{\vp} = \vp = b = \psi \leq 0 & \mbox{on $\de\Omega$}.
\end{cases}
\end{equation}
Set $K_{0}:=\sup_{\ov\Omega}|\underline{\vp}|+1$ and then $\sup_{\ov\Omega}|\vp|\leq K_{0}$.

\subsection{Global logarithmic sup-convolution approximation}

For $t>0$, denote
\[
\Omega_{\ve} := \{ x\in\Omega: d(x,\de\Omega)>t\}
\]
For $\ve\in(0,1)$ and $x\in\Omega_{2\ve^{1/8}}$, define
\begin{equation}\label{def of vp ve boundary}
\vp_{\ve}(x) := \sup_{\xi\in T_{x}X,\,\exp_{x}\xi\in\ov{\Omega}}\Big\{\vp(\exp_{x}\xi)-\Psi_{\ve}(|\xi|_{x}^{2})\Big\}.
\end{equation}
By the same arguments of Lemma \ref{elementary properties} and Proposition \ref{k-subharmonicity and L1 approximation}, we have the following.

\begin{lemma}\label{elementary properties boundary}
The following holds:
\begin{enumerate}\setlength{\itemsep}{1mm}
\item[(i)] $-K_{0}\leq\vp\leq \vp_{\ve}\leq0$ on $\Omega_{2\ve^{1/8}}$;
\item[(ii)] for any $x\in\Omega_{2\ve^{1/8}}$, each maximizing vector $\xi_{x}$ in \eqref{def of vp ve boundary} satisfies $|\xi_{x}|_{x}^{2}\leq \ve$;
\item[(iii)] if $x\in\Omega_{2\ve^{1/8}}$ and $d(x,y)\leq\ve$, then $\vp(y)\leq \vp_{\ve}(x)+2K_{0}(-\log\ve)^{-1}$.
\end{enumerate}
\end{lemma}

\begin{proposition}\label{k-subharmonicity and L1 approximation boundary}
There exist constants $\ve_{0}\in(0,1)$ and $C$ depending only on $k$, $n$ and $(\Omega,\omega)$ such that the following holds. For $\ve\in(0,\ve_{0})$, one has
\begin{enumerate}\setlength{\itemsep}{1mm}
\item[(i)] $(1+K_{0}\theta_{\ve})\omega+\ddbar\vp_{\ve}\in\Gamma_{k}(\Omega_{2\ve^{1/8}},\omega)$ in the viscosity sense, where $\theta_{\ve}=C(-\log\ve)^{-1}$;
\item[(ii)] $\|\vp_{\ve}-\vp\|_{L^{1}(\Omega_{2\ve^{1/8}})}\leq CK_{0}\ve^{1/4}$.
\end{enumerate}
\end{proposition}

For any $x\in\Omega_{2\ve^{1/8}}\setminus\ov{\Omega}_{3\ve^{1/8}}$, we have $d(x,\de\Omega)\in(2\ve^{1/8},3\ve^{1/8})$, there exists a point $p_{x}\in\de\Omega$ such that $d(x,p_{x})<3\ve^{1/8}$. Let $\xi_{x}$ be the maximizing vector in \eqref{def of vp ve boundary}. By Lemma \ref{elementary properties boundary} (ii), we see that $d(\exp_{x}\xi_{x},p_{x})<3\ve^{1/8}+\ve^{1/2}<4\ve^{1/8}$ and
\[
0 \leq \vp_{\ve}(x)-\vp(x) \leq \vp(\exp_{x}\xi_{x})-\vp(x)
\leq \vp(\exp_{x}\xi_{x})-\psi(p_{x})+\psi(p_{x})-\vp(x).
\]
Combining this with \eqref{underline vp vp b},
\[
\begin{split}
& \vp(\exp_{x}\xi_{x})-\psi(p_{x})+\psi(p_{x})-\vp(x) \\[1mm]
\leq {} & b(\exp_{x}\xi_{x})-\psi(p_{x})+\psi(p_{x})-\underline{\vp}(x) \\[1mm]
= {} & b(\exp_{x}\xi_{x})-b(p_{x})+\underline{\vp}(p_{x})-\underline{\vp}(x) \\
< {} & (\sup_{\ov\Omega}|\de b|+\sup_{\ov\Omega}|\de \underline{\vp}|)\cdot 4\ve^{1/8}.
\end{split}
\]
Set $C_{\ve}:=(\sup_{\ov\Omega}|\de b|+\sup_{\ov\Omega}|\de \underline{\vp}|)\cdot 4\ve^{1/8}$. Then the above shows
\begin{equation}\label{vp ve C ve vp}
\vp_{\ve}(x)-C_{\ve} < \vp(x) \ \ \text{for $x\in\Omega_{2\ve^{1/8}}\setminus\ov{\Omega}_{3\ve^{1/8}}$}.
\end{equation}
Define the global logarithmic sup-convolution approximation by
\[
\ti{\vp}_\e(x) :=
\begin{cases}
\ \max\{\vp(x),\vp_{\ve}(x)-C_{\ve}\} & \mbox{if $x\in\Omega_{2\ve^{1/8}}$}, \\[1mm]
\ \vp(x) & \mbox{if $x\in\ov{\Omega}\setminus\Omega_{2\ve^{1/8}}$}, \\
\end{cases}
\]

\begin{proposition}\label{global log sup convolution}
There exist constants $\ve_{0}\in(0,1)$ and $C$ depending only on $k$, $n$ and $(\Omega,\omega)$ such that the following holds. For $\ve\in(0,\ve_{0})$, one has
\begin{enumerate}\setlength{\itemsep}{1mm}
\item[(i)] $(1+K_{0}\theta_{\ve})\omega+\ddbar\ti\vp_{\ve}\in\Gamma_{k}(\Omega,\omega)$ in the viscosity sense, where $\theta_{\ve}=C(-\log\ve)^{-1}$;
\item[(ii)] $\|\ti\vp_{\ve}-\vp\|_{L^{1}(\Omega)}\leq CK_{0}\ve^{1/4}$.
\end{enumerate}
\end{proposition}

\begin{proof}
For (i), by \eqref{vp ve C ve vp}, we see that $\ti{\vp}_\e\in C^{0}(\ov{\Omega})$. Then (i) follows from Proposition \eqref{k-subharmonicity and L1 approximation boundary} (i). Note that
\[
0 \leq \ti{\vp}_{\ve}-\vp = (\vp_{\ve}-C_{\ve}-\vp)_{+} \leq (\vp_{\ve}-\vp)_{+} = \vp_{\ve}-\vp \ \ \text{on $\Omega_{2\ve^{1/8}}$}.
\]
Together with Proposition \eqref{k-subharmonicity and L1 approximation boundary} (ii), we obtain
\[
\|\ti{\vp}_{\ve}-\vp\|_{L^{1}(\Omega)} = \|\ti{\vp}_{\ve}-\vp\|_{L^{1}(\Omega_{2\ve^{1/8}})}
\leq \|\vp_{\ve}-\vp\|_{L^{1}(\Omega_{2\ve^{1/8}})} \leq CK_{0}\ve^{1/4}.
\]
\end{proof}

\subsection{Stability estimate}

\begin{theorem}\label{stability estimate boundary}
Let $(\Omega,\omega)$ be an $n$-dimensional compact K\"ahler manifold with $\de\Omega\neq\emptyset$, $\vp$ be a smooth solution and $\underline{\vp}$ be a subsolution. Suppose that $v\in C(\ov{\Omega})$ satisfies
\[
v \leq 0 \ \ \text{in $\Omega$}, \ \ \ \ \ v\leq \vp \ \ \text{on $\de\Omega$},
\]
and
\[
\left(1+\frac{\delta}{2}\right)\omega+\ddbar v \in \Gamma_{k}(\Omega,\omega) \ \ \text{in the viscosity sense for some $\delta>0$}.
\]
For $p>n/k$ and $\beta\in\left(0,\frac{kp-n}{nkp+kp-n}\right)$. There exists a constant $C$ depending only on $\beta$, $p$, $k$, $n$ and $(\Omega,\omega)$ such that
\[
\sup_{X}(v-\vp) \leq 4\delta\|v\|_{L^{\infty}}+C\|v\|_{L^{\infty}}^{1-\beta}\cdot\|f\|_{L^{p}}^{\frac{1-\beta}{k}}\cdot\|(v-\vp)_{+}\|_{L^{1}}^{\beta}.
\]
\end{theorem}
\begin{proof}
    The argument adapts the local auxiliary Monge--Amp\`ere construction of Guo--Phong \cite[Theorem 2]{GP24} and the viscosity comparison of Cheng--Xu \cite[Proposition 3.1]{CX24}. Set
\[
V := \sup_{\ov\Omega}|v|+\sup_{\bar\Omega}\vp+1, \ \ \ w := v-\vp, \ \ \ W = \sup_{\ov\Omega}w.
\]
If $W\leq 4V\delta$, then we are done. Next assume that $W>4V\delta$ and define
\[
\tau := \frac{W}{4V}, \ \ \
\ti{w} := (1-\tau)v-\vp, \ \ \
\ti{W} := \sup_{\ov\Omega}\ti{w}.
\]
Since $0\leq-v\leq V$ in $\Omega$ and $w\leq0$ on $\de\Omega$, then
\begin{equation}\label{ti w de Omega}
\ti{w} = w-\tau v \geq w \ \ \text{in $\Omega$}, \ \ \ \ \
\ti{w} = w-\tau v \leq \tau V = \frac{W}{4} \ \ \text{on $\de\Omega$}.
\end{equation}
It follows that $\ti{w}$ attains its maximum at an interior point $x_{0}\in\Omega$. Choose a coordinate ball $B_{R}$ centered at $x_{0}$. Note that $B_{R}$ may intersect $\de\Omega$. Denote
\[
D := B_{R}\cap\Omega, \ \ \
\omega_{\mathrm{E}} := \ddbar|z|^{2}.
\]
and assume (shrinking $B_{R}$ if necessary)
\[
\frac{1}{2}\omega_{\mathrm{E}} \leq \omega \leq 2\omega_{\mathrm{E}} \ \ \text{in $D$}.
\]
Choose sufficiently small constant $a\in(0,1/2)$ such that
\[
a\omega_{\mathrm{E}} \leq \frac{1}{4}\omega, \ \ \
aR^{2} \leq 1, \ \ \
c := \frac{aR^{2}}{2}.
\]
Define the function $\Phi$ on $D$ by
\[
\Phi := \ti{w}-\ti{W}+c\tau-a\tau|z|^{2} \ \ \text{on $D$}
\]
and the set $S$ by
\[
S := \{x\in\Omega: w(x)>W/2\}.
\]
We split the following argument into three steps.

\bigskip
\noindent
{\bf Step 1.} $\{x\in D:\Phi(x)\geq0\}\neq\emptyset$ and $\{x\in D:\Phi(x)\geq0\}\Subset D\cap S$.
\bigskip

For convenience, write $E:=\{x\in D:\Phi(x)\geq0\}$. Since $z(x_{0})=0$, then $\Phi(x_{0})=c\tau>0$ and so $E\neq\emptyset$.
By the definitions of $\Phi$, $a$, $c$ and $\tau$, and \eqref{ti w de Omega}, we have
\[
\Phi \leq c\tau-a\tau R^{2} = -\frac{1}{2}a\tau R^{2} \ \ \text{on $\de B_{R}\cap\Omega$},
\]
\[
\Phi \leq \frac{W}{4}-\ti{W}+c\tau \leq -\frac{3W}{4}+c\tau \leq -\frac{W}{2} \ \ \text{on $B_{R}\cap\de \Omega$}.
\]
This shows $\Phi<0$ on $\de D$ and so $E\Subset D$.

On the other hand, for $x\in E$, we have $\Phi(x)\geq0$ and then
\[
\ti{w}(x) \geq \ti{W}-c\tau \geq W-c\tau.
\]
Together with $w=\ti{w}+\tau v$, $v\geq-V$ and $c\leq 1/2\leq V/2$,
\[
w(x) \geq W-c\tau-\tau V \geq W-\frac{3}{2}\tau V = \frac{5W}{8}.
\]
This shows $x\in S$ and so $\ov{E}=E\subset S$. Together with $E\Subset D$, we obtain $E\Subset D\cap S$.

\medskip

By Step 1, we may choose $\eta\in C_{c}^{\infty}(D\cap S)$ such that
\[
0 \leq \eta \leq 1 \ \text{in $D\cap S$}, \ \ \
\eta = 1 \ \text{near $\{\Phi\geq0\}$}.
\]
Define
\[
G :=
\begin{cases}
\ \eta \cdot f^{n/k}\cdot\frac{\omega^{n}}{\omega_{\mathrm{E}}^{n}} & \mbox{in $B_{R}\cap D$}, \\[1mm]
\ 0 & \mbox{in $B_{R}\setminus D$}.
\end{cases}
\]
Note that $G$ is smooth on $B_{R}$. By \cite[Theorem 1.1]{CKNS85}, the following Dirichlet problem admits a smooth solution:
\begin{equation}\label{CMAE B R}
\begin{cases}
\ (\ddbar u_{\ve})^{n} = (G+\ve)\omega_{\mathrm{E}}^{n} & \mbox{in $B_{R}$}, \\[1.2mm]
\ \ddbar u_{\ve} > 0 & \mbox{in $B_{R}$}, \\[1mm]
\ u_{\ve}= 0 & \mbox{on $\de B_{R}$}.
\end{cases}
\end{equation}

\bigskip
\noindent
{\bf Step 2.} $\Phi+2u_{\ve}\leq0$ on $D$.
\bigskip

Let $y_{0}$ be the maximum point of $\Phi+2u_{\ve}$ in $\ov{D}$. If $y_{0}\in\de D$ or $\Phi(y_{0})\leq0$, then we are done. Next we assume that $y_{0}$ is an interior point of $D$ and $\Phi(y_{0})>0$. It then follows from Step 1 that $\eta(y_{0})=1$. Set $M:=(\Phi+2u_{\ve})(y_{0})$ and
\[
P := \frac{\vp+\ti{W}-c\tau+a\tau|z|^{2}-2u_{\ve}+M}{1-\tau}.
\]
It is clear that $P$ is a upper test function of $v$ at $y_{0}$ and so
\begin{equation}\label{test P}
\lambda\left[\left(1+\frac{\delta}{2}\right)\omega+\ddbar P\right] \in \Gamma_{k}.
\end{equation}
Denote
\[
\mathcal{A} := (1-\tau)\left[\left(1+\frac{\delta}{2}\right)\omega+\ddbar P \right], \ \ \
\theta := \tau-\frac{(1-\tau)\delta}{2}.
\]
It follows from $\tau=W/(4V)>\delta$ that $\theta\geq\tau/2$. Direct calculation and $a\omega_{\mathrm{E}}\leq\omega/4$ show
\[
\omega+\ddbar\vp = \mathcal{A}+\theta\omega-a\tau\omega_{\mathrm{E}}+2\ddbar u_{\ve}
\geq \mathcal{A}+2\ddbar u_{\ve}.
\]
By \eqref{CHE Dirichlet} and \eqref{test P},
\[
f^{1/k} = \sigma_{k}^{1/k}(\omega+\ddbar\vp) \geq 2\sigma_{k}^{1/k}(\ddbar u_{\ve}).
\]
However, from Maclaurin's inequality, \eqref{CMAE B R} and $\eta(y_{0})=1$,
\[
2\sigma_{k}^{1/k}(\ddbar u_{\ve}) \geq 2\sigma_{n}^{1/n}(\ddbar u_{\ve}) > 2f^{1/k},
\]
which is impossible.

\bigskip
\noindent
{\bf Step 3.} Completion of the proof.
\bigskip

By Step 2, $\Phi(x_{0})=c\tau$ and $\tau=W/(4V)$,
\begin{equation}\label{W upper bound}
\frac{cW}{4V} = c\tau = \Phi(x_{0}) \leq 2\|u_{\ve}\|_{L^{\infty}(B_{R})}.
\end{equation}
For any $r\in(1,kp/n)$, by \cite[Theorem 3]{Kol96} and \cite[Remark 2]{Blocki11} (see also \cite[Theorem 1.2]{WWZ21}), one has
\[
\|u_{\ve}\|_{L^{\infty}(B_{R})} \leq C_{r}\|G+\ve\|_{L^{r}(B_{R},\omega_{\mathrm{E}})}^{1/n}.
\]
Letting $\ve\to0$ and using \eqref{W upper bound},
\[
W \leq C_{r}V\left(\int_{S}f^{nr/k}\omega^{n}\right)^{\frac{1}{nr}}.
\]
Set $\alpha_{r}:=\frac{1}{nr}-\frac{1}{kp}$. By Holder's inequality,
\[
W \leq C_{r}V\cdot\|f\|_{L^{p}}^{1/k}\cdot\Vol_{\omega}^{\alpha_{r}}(S).
\]
Recalling $S=\{\Phi\geq W/2\}$, then
\[
\Vol_{\omega}(S) \leq \frac{2}{W}\|w_{+}\|_{L^{1}}.
\]
Denote $\beta:=\frac{\alpha_{r}}{1+\alpha_{r}}$. We obtain
\[
W \leq \left(C_{r}V\cdot\|f\|_{L^{p}}^{1/k}\right)^{\frac{1}{1+\alpha_{r}}}\|w_{+}\|_{L^{1}}^{\frac{\alpha_{r}}{1+\alpha_{r}}}
= C_{r}V^{1-\beta}\cdot\|f\|_{L^{p}}^{\frac{1-\beta}{k}}\cdot\|w_{+}\|_{L^{1}}^{\beta},
\]
as required.
\end{proof}

\appendix

\section{Proof of Theorem \ref{HMW estimate}}\label{App:proof of Hou-Ma-Wu's estimate}

In this appendix, we give the proof of Theorem \ref{HMW estimate}. As mentioned before, the argument mainly follows \cite[Theorem 1.1]{HMW10} with some slight modifications (using the test function of Xu \cite[Proposition 5.2]{Xu26}).

\begin{proof}[Proof of Theorem \ref{HMW estimate}]
For notational convenience, set
\[
M_{0} := \sup_{X}|\vp|+1, \ \ \
K_{f} := \sup_{X}f^{1/k}+\sup_{X}|\de f^{1/k}|^{2}+\sup_{X}|\de\dbar f^{1/k}|+1.
\]
Without loss of generality, we assume that $C_{\BK}\geq1$. It then suffices to show
\begin{equation}\label{HMW estimate main}
\sup_{X}|\de\dbar \vp| \leq e^{A_{n}C_{\BK}M_{0}}\left(\sup_{X}|\de\vp|^{2}+\|f\|_{L^{\infty}}+K_{f}\right).
\end{equation}
In the following argument, we often use $A_{n}$ to denote a dimensional constant (depending only on $n$), which may differ from line to line.

For $x\in X$ and unit vector $v\in T_{x}X$, consider the quantity
\[
Q(x,v) := \log(1+\vp_{i\ov{j}}v^{i}\ov{v}^{j})+\rho_{1}(|\de\vp|^{2})+\rho_{0}(\vp),
\]
where
\[
\rho_{1}(t) := -\frac{1}{2}\log\left(1-\frac{t}{2K}\right), \ \
K := \sup_{X}|\de\vp|^{2}+\|f\|_{L^{\infty}}+K_{f}
\]
and

\[
\rho_{0}(t) := -\frac{1}{4\delta}\log\left(1+\frac{t}{2M_{0}}\right), \ \
M_{0} := \sup_{X}|\vp|+1, \ \ \delta := \Big(96C_{\BK}M_{0}\Big)^{-1}.
\]
Direct calculation shows
\begin{equation}\label{rho 1}
0 \leq \rho_{1} \leq \frac{1}{2}\log 2, \ \ \
\frac{1}{4K} \leq \rho_{1}' \leq \frac{1}{2K}, \ \ \
\rho_{1}'' = 2(\rho_{1}')^{2}
\end{equation}
and
\begin{equation}\label{rho 0}
-\frac{\log(3/2)}{4\delta} \leq \rho_{0} \leq 0, \ \ \
8C_{\BK} \leq -\rho_{0}' \leq 12C_{\BK}, \ \ \
\rho_{0}'' = 4\delta(\rho_{0}')^{2}.
\end{equation}
Suppose that $Q$ attains its maximum at $(x_{0},v_{0})$. Since $(X,\omega)$ is K\"ahler, there exists a holomorphic coordinate $\{z^{i}\}_{i=1}^{n}$ centered at $x_{0}$ such that
\[
g_{i\ov{j}} = \delta_{ij}, \ \ \
\frac{\de g_{i\ov{j}}}{\de z^{k}} = 0, \ \ \
\vp_{i\ov{j}} = \vp_{i\ov{i}}\delta_{ij}, \ \ \
\vp_{1\ov{1}} \geq \cdots \geq \vp_{n\ov{n}} \ \ \ \text{at $x_{0}$}.
\]
Then we have $v_{0}=\de/\de z^{1}$ at $x_{0}$. Near $x_{0}$, define a local vector field by
\[
w := g_{1\ov{1}}^{-1/2}\frac{\de}{\de z^{1}}
\]
and a local quantity by
\[
\hat{Q}(x) := Q(x,w) = \log(1+g_{1\ov{1}}^{-1}\vp_{1\ov{1}})+\rho_{1}(|\de\vp|^{2})+\rho_{0}(\vp).
\]
It is clear that $\hat{Q}$ attains its maximum at $x_{0}$. Let $F^{i\ov{j}}$ and $F^{i\ov{j},p\ov{q}}$ be the first and second differentiation of the operator $\sigma_{k}^{1/k}$. By the maximum principle, at $x_{0}$, we have
\begin{equation}\label{hat Q i}
0 = \hat{Q}_{i} = \frac{\vp_{1\ov{1}i}}{1+\vp_{i\ov{i}}}+\rho_{1}'(|\de\vp|^{2})_{i}+\rho_{0}'\vp_{i}
\end{equation}
and
\begin{equation}\label{F ii hat Q ii}
\begin{split}
0 \geq F^{i\ov{i}}\hat{Q}_{i\ov{i}}
& = \frac{F^{i\ov{i}}\vp_{1\ov{1}i\ov{i}}}{1+\vp_{1\ov{1}}}-\frac{F^{i\ov{i}}|\vp_{1\ov{1}i}|^{2}}{(1+\vp_{1\ov{1}})^{2}} \\[0.6mm]
& \ +\rho_{1}'F^{i\ov{i}}(|\de\vp|^{2})_{i\ov{i}}+\rho_{1}''F^{i\ov{i}}|(|\de\vp|^{2})_{i}|^{2} \\[1.2mm]
& \ +\rho_{0}'F^{i\ov{i}}\vp_{i\ov{i}}+\rho_{0}''F^{i\ov{i}}|\vp_{i}|^{2}. \\[0.5mm]
\end{split}
\end{equation}
Set $\lambda_{i}:=1+\vp_{i\ov{i}}$. Using $Q(x,v)\leq Q(x_{0},v_{0})$, \eqref{rho 1} and \eqref{rho 0}, we see that
\[
\log\left(1+\sup_{X}|\de\dbar\vp|\right) \leq \log\lambda_{1}+\frac{1}{2}\log2+\frac{\log3/2}{4\delta}
\]
and so
\[
\sup_{X}|\de\dbar\vp| \leq \lambda_{1} \, \exp\left(\frac{1}{2}\log2+\frac{\log3/2}{4\delta}\right) \leq \lambda_{1}e^{A_{n}C_{\BK}M_{0}}.
\]
To prove \eqref{HMW estimate main}, it suffices to show
\begin{equation}\label{HMW estimate goal}
\lambda_{1} \leq e^{A_{n}C_{\BK}M_{0}}K.
\end{equation}
Set $\mathcal{F}:=\sum_{i}F^{i\ov{i}}$. It follows from Maclaurin's inequlaity that $\mathcal{F}\geq1$ (see \cite[p.555]{HMW10}). Following the same calculation of \cite[p.552-554]{HMW10}, and using \eqref{rho 1} and \eqref{rho 0}, we have
\[
\begin{split}
F^{i\ov{i}}\vp_{1\ov{1}i\ov{i}}
\geq {} & (f^{1/k})_{1\ov{1}}-F^{i\ov{j},p\ov{q}}\vp_{i\ov{j}1}\vp_{p\ov{q}\ov{1}}+\sum_{i}F^{i\ov{i}}(\vp_{1\ov{1}}-\vp_{i\ov{i}})R_{1\ov{1}i\ov{i}} \\[-1.6mm]
\geq {} & -K_{f}-F^{i\ov{j},p\ov{q}}\vp_{i\ov{j}1}\vp_{p\ov{q}\ov{1}}-C_{\BK}(\lambda_{1}\mathcal{F}-f^{1/k}) \\[1.6mm]
\geq {} & -K_{f}-F^{i\ov{j},p\ov{q}}\vp_{i\ov{j}1}\vp_{p\ov{q}\ov{1}}-C_{\BK}\lambda_{1}\mathcal{F},
\end{split}
\]

\[
\begin{split}
\rho_{1}'F^{i\ov{i}}(|\de\vp|^{2})_{i\ov{i}}
= {} & \rho_{1}'F^{i\ov{i}}\vp_{i\ov{i}}^{2}+\rho_{1}'\sum_{i,j}F^{i\ov{i}}|\vp_{ij}|^{2}+\rho_{1}'F^{i\ov{i}}(\vp_{pi\ov{i}}\vp_{\ov{p}}+\vp_{\ov{p}i\ov{i}}\vp_{p}) \\[0.6mm]
\geq {} & \rho_{1}'F^{i\ov{i}}\lambda_{i}^{2}-2\rho_{1}'f^{1/k}+\rho_{1}'\mathcal{F}+\rho_{1}'\sum_{i,j}F^{i\ov{i}}|\vp_{ij}|^{2} \\
& -\rho_{1}'|\de f^{1/k}|^{2}-\rho_{1}'|\de\vp|^{2}-C_{\BK}\rho_{1}'|\de\vp|^{2}\mathcal{F} \\[1.6mm]
\geq {} & \frac{1}{4K}F^{i\ov{i}}\lambda_{i}^{2}-C_{\BK}\mathcal{F}-4
\end{split}
\]
and
\[
\rho_{0}'F^{i\ov{i}}\vp_{i\ov{i}} = \rho_{0}'F^{i\ov{i}}(\lambda_{i}-1)  = \rho_{0}'f^{1/k}-\rho_{0}'\mathcal{F} \geq -12C_{\BK}f^{1/k}+8C_{\BK}\mathcal{F} .
\]
Substituting the above into \eqref{F ii hat Q ii},
\[
\begin{split}
0 \geq {} & -\frac{F^{i\ov{j},p\ov{q}}\vp_{i\ov{j}1}\vp_{p\ov{q}\ov{1}}}{\lambda_{1}}-\frac{F^{i\ov{i}}|\vp_{1\ov{1}i}|^{2}}{\lambda_{1}^{2}} +\frac{1}{4K}F^{i\ov{i}}\lambda_{i}^{2}+\rho_{1}''F^{i\ov{i}}|(|\de\vp|^{2})_{i}|^{2} \\
& +\rho_{0}''F^{i\ov{i}}|\vp_{i}|^{2}+6C_{\BK}\mathcal{F}-\frac{K_{f}}{\lambda_{1}}-4-12C_{\BK}f^{1/k}.
\end{split}
\]
Without loss of generality, we may assume that $\lambda_{1}\geq K_{f}$. Using $\mathcal{F}\geq1$ and $C_{\BK}\geq1$, we obtain
\begin{equation}\label{HMW estimate eqn 1}
\begin{split}
0 \geq {} & -\frac{F^{i\ov{j},p\ov{q}}\vp_{i\ov{j}1}\vp_{p\ov{q}\ov{1}}}{\lambda_{1}}-\frac{F^{i\ov{i}}|\vp_{1\ov{1}i}|^{2}}{\lambda_{1}^{2}} +\frac{1}{4K}F^{i\ov{i}}\lambda_{i}^{2}+\rho_{1}''F^{i\ov{i}}|(|\de\vp|^{2})_{i}|^{2} \\[1mm]
& +\rho_{0}''F^{i\ov{i}}|\vp_{i}|^{2}+C_{\BK}\mathcal{F}-12C_{\BK}f^{1/k}.
\end{split}
\end{equation}
Now the task is to deal with the third order terms:
\[
\mathcal{B} := \frac{F^{i\ov{i}}|\vp_{1\ov{1}i}|^{2}}{\lambda_{1}^{2}}, \ \ \
\mathcal{G} := -\frac{F^{i\ov{j},p\ov{q}}\vp_{i\ov{j}1}\vp_{p\ov{q}\ov{1}}}{\lambda_{1}}.
\]
We split the argument into two cases.

\bigskip
\noindent
{\bf Case 1.} $\lambda_{n}<-\delta\lambda_{1}$.
\bigskip

By \eqref{hat Q i}, \eqref{rho 1} and \eqref{rho 0},
\[
\begin{split}
\mathcal{B}
\leq {} & 2(\rho_{1}')^{2}F^{i\ov{i}}|(|\de\vp|^{2})_{i}|^{2}+2(\rho_{0}')^{2}F^{i\ov{i}}|\vp_{i}|^{2}
\leq {}  \rho_{1}''F^{i\ov{i}}|(|\de\vp|^{2})_{i}|^{2}+A_{n}C_{\BK}^{2}K\mathcal{F}.
\end{split}
\]
Substituting this into \eqref{HMW estimate eqn 1}, and using $\mathcal{G}\geq0$ and $\rho_{0}''\geq0$,
\[
0 \geq \frac{1}{4K}F^{i\ov{i}}\lambda_{i}^{2}-A_{n}C_{\BK}^{2}K\mathcal{F}-12C_{\BK}f^{1/k}
\geq \frac{1}{4K}F^{i\ov{i}}\lambda_{i}^{2}-A_{n}C_{\BK}^{2}K\mathcal{F}-12C_{\BK}K.
\]
Together with $\mathcal{F}\geq1$, $\lambda_{n}<-\delta\lambda_{1}$ and $F^{n\ov{n}}\geq\mathcal{F}/n$,
\[
\begin{split}
0 \geq {} & \frac{1}{4K}F^{n\ov{n}}\lambda_{n}^{2}-A_{n}C_{\BK}^{2}K\mathcal{F}-12C_{\BK}K\mathcal{F}
\geq {}  \frac{\delta^{2}}{4nK}\lambda_{1}^{2}\mathcal{F}-(A_{n}C_{\BK}^{2}+12C_{\BK})K\mathcal{F}.
\end{split}
\]
It follows that $\lambda_{1}\leq A_{n}\delta^{-1}C_{\BK}K$. Recalling $\delta=(96C_{\BK}M_{0})^{-1}$, we obtain
\[
\lambda_{1} \leq A_{n}C_{\BK}^{2}M_{0}K,
\]
which implies \eqref{HMW estimate goal}.

\bigskip
\noindent
{\bf Case 2.} $\lambda_{n}\geq-\delta\lambda_{1}$.
\bigskip

Define the index set
\[
I := \{i:F^{i\ov{i}} > \delta^{-1}F^{1\ov{1}}\}.
\]
The term $\mathcal{B}$ can be decomposed into three parts:
\[
\mathcal{B} = \sum_{i\notin I}\frac{F^{i\ov{i}}|\vp_{1\ov{1}i}|^{2}}{\lambda_{1}^{2}}
+2\delta\sum_{i\in I}\frac{F^{i\ov{i}}|\vp_{1\ov{1}i}|^{2}}{\lambda_{1}^{2}}
+(1-2\delta)\sum_{i\in I}\frac{F^{i\ov{i}}|\vp_{1\ov{1}i}|^{2}}{\lambda_{1}^{2}}.
\]
Using \eqref{hat Q i}, \eqref{rho 1} and \eqref{rho 0},
\[
\begin{split}
\sum_{i\notin I}\frac{F^{i\ov{i}}|\vp_{1\ov{1}i}|^{2}}{\lambda_{1}^{2}}
\leq {} & 2(\rho_{1}')^{2}\sum_{i\notin I}F^{i\ov{i}}|(|\de\vp|^{2})_{i}|^{2}+2(\rho_{0}')^{2}\sum_{i\notin I}F^{i\ov{i}}|\vp_{i}|^{2} \\
\leq {} & \rho_{1}''\sum_{i\notin I}F^{i\ov{i}}|(|\de\vp|^{2})_{i}|^{2}+2(\rho_{0}')^{2}\delta^{-1}KF^{1\ov{1}} \\
\leq {} & \rho_{1}''\sum_{i\notin I}F^{i\ov{i}}|(|\de\vp|^{2})_{i}|^{2}+A_{n}C_{\BK}^{3}M_{0}KF^{1\ov{1}},
\end{split}
\]

\[
\begin{split}
2\delta\sum_{i\in I}\frac{F^{i\ov{i}}|\vp_{1\ov{1}i}|^{2}}{\lambda_{1}^{2}}
\leq {} & 4\delta(\rho_{1}')^{2}\sum_{i\in I}F^{i\ov{i}}|(|\de\vp|^{2})_{i}|^{2}+4\delta(\rho_{0}')^{2}\sum_{i\in I}F^{i\ov{i}}|\vp_{i}|^{2} \\
\leq {} & \rho_{1}''\sum_{i\in I}F^{i\ov{i}}|(|\de\vp|^{2})_{i}|^{2}+\rho_{0}''\sum_{i\in I}F^{i\ov{i}}|\vp_{i}|^{2}
\end{split}
\]
and
\[
\begin{split}
\mathcal{G}
\geq {} & -\sum_{i\in I}\frac{F^{i\ov{1},1\ov{i}}|\vp_{1\ov{1}i}|^{2}}{\lambda_{1}}
= \sum_{i\in I}\frac{(F^{i\ov{i}}-F^{1\ov{1}})|\vp_{1\ov{1}i}|^{2}}{\lambda_{1}(\lambda_{1}-\lambda_{i})} \\
\geq {} & \frac{1-\delta}{1+\delta}\sum_{i\in I}\frac{F^{i\ov{i}}|\vp_{1\ov{1}i}|^{2}}{\lambda_{1}^{2}}
\geq (1-2\delta)\sum_{i\in I}\frac{F^{i\ov{i}}|\vp_{1\ov{1}i}|^{2}}{\lambda_{1}^{2}}.
\end{split}
\]
Substituting the above into \eqref{HMW estimate eqn 1},
\[
\begin{split}
0 \geq {} & \frac{1}{4K}F^{i\ov{i}}\lambda_{i}^{2}-A_{n}C_{\BK}^{3}M_{0}KF^{1\ov{1}}+C_{\BK}\mathcal{F}-12C_{\BK}f^{1/k} \\
\geq {} & \frac{1}{4K}F^{1\ov{1}}\lambda_{1}^{2}-A_{n}C_{\BK}^{3}M_{0}KF^{1\ov{1}}+C_{\BK}\mathcal{F}-12C_{\BK}f^{1/k}.
\end{split}
\]
Without loss of generality, we assume that $\lambda_{1}\geq 10A_{n}C_{\BK}^{2}M_{0}K$, and then
\[
\frac{1}{8K}F^{1\ov{1}}\lambda_{1}^{2}+C_{\BK}\mathcal{F} \leq 12C_{\BK}f^{1/k}.
\]
By the same argument of \cite[(2.32)]{HMW10} and $\|f\|_{L^{\infty}}\leq K$, we see that
\[
\lambda_{1} \leq (A_{n}C_{\BK}Kf)^{1/2}
\leq A_{n}C_{\BK}(K+\|f\|_{L^{\infty}})
\leq A_{n}C_{\BK}K,
\]
which implies \eqref{HMW estimate goal}.
\end{proof}


\begin{thebibliography}{99}

\bibitem{Blocki09} B\l ocki, Z. {\em A gradient estimate in the Calabi-Yau theorem}, Math. Ann. {\bf 344} (2009), no. 2, 317--327.

\bibitem{Blocki11} B\l ocki, Z. {\em On the uniform estimate in the Calabi-Yau theorem, II}, Sci. China Math. {\bf 54} (2011), no. 7, 1375--1377.

\bibitem{CKNS85} Caffarelli, L.; Kohn, J. J.; Nirenberg, L.; Spruck, J. {\em The Dirichlet problem for nonlinear second-order elliptic equations. II. Complex Monge-Amp\`ere, and uniformly elliptic, equations}, Comm. Pure Appl. Math. {\bf 38} (1985), no. 2, 209--252.

\bibitem{Calabi57} Calabi, E. {\em On K\"ahler manifolds with vanishing canonical class}, Algebraic geometry and topology. A symposium in honor of S. Lefschetz, pp. 78--89, Princeton University Press, Princeton, NJ, 1957.

\bibitem{CX24} Cheng, J.; Xu, Y. {\em  Regularization of $m$-subharmonic functions and H\"older continuity}, Math. Res. Lett. {\bf 31} (2024), no. 4, 951--984.

\bibitem{CLM26} Chu, J; Liu, Y; McCleerey, N. {\em The eigenvalue problem for the complex Hessian operator on $m$-pseudoconvex manifolds}, J. Funct. Anal. {\bf 290} (2026), no. 3, Paper No. 111258, 58 pp.

\bibitem{CLMZ26} Chu, J; Liu, Y; McCleerey, N; Zhang, W.
{\em Some variational problems for the complex Monge--Amp\`ere operator}, preprint, arXiv:2604.13421.

\bibitem{CP22} Collins, T.; Picard, S. {\em The Dirichlet problem for the $k$-Hessian equation on a complex manifold}, Amer. J. Math. {\bf 144} (2022), no. 6, 1641--1680

\bibitem{DK14} Dinew, S.; Ko\l odziej, S. {\em A priori estimates for complex Hessian equations}, Anal. PDE {\bf 7} (2014), no. 1, 227--244.

\bibitem{DK17} Dinew, S.; Ko\l odziej, S. {\em Liouville and Calabi-Yau type theorems for complex Hessian equations}, Amer. J. Math. {\bf 139} (2017), no. 2, 403--415.

\bibitem{Guan} Guan, P. {\em On pointwise gradient estimates for the complex Monge-Amp\'ere equation}, preprint.

\bibitem{GP24} Guo, B.; Phong, D. H. {\em On $L^{\infty}$ estimates for fully non-linear partial differential equations}, Ann. of Math. (2) {\bf 200} (2024), no. 1, 365--398.

\bibitem{Hou09} Hou, Z. {\em Complex Hessian equation on K\"ahler manifold}, Int. Math. Res. Not. IMRN {\bf 2009}, no. 16, 3098--3111.

\bibitem{HMW10} Hou, Z.; Ma, X.-N.; Wu, D. {\em A second order estimate for complex Hessian equations on a compact K\"ahler manifold}, Math. Res. Lett. {\bf 17} (2010), no. 3, 547--561.

\bibitem{Kol96} Ko\l odziej, S. {\em Some sufficient conditions for solvability of the Dirichlet problem for the complex Monge-Amp\`re operator}, Ann. Polon. Math. {\bf 65} (1996), no. 1, 11--21.

\bibitem{WWZ21} Wang, J.; Wang, X.-J.; Zhou, B. {\em A priori estimate for the complex Monge-Amp\`ere equation}, Peking Math. J. {\bf 4} (2021), no. 1, 143--157.

\bibitem{Xu26} Xu, Y. {\em Regularization and H\"older continuity for complex Hessian equations on Hermitian manifolds}, preprint, arXiv:2608.15552.

\bibitem{Yau78} Yau, S.-T. {\em On the Ricci curvature of a compact K\"ahler manifold and the complex Monge-Amp\'ere equation. I}, Comm. Pure Appl. Math. {\bf 31} (1978), no. 3, 339--411.

\end{thebibliography}
\end{document}